\documentclass[reqno]{amsart}

\usepackage[colorlinks=true, pdfstartview=FitV, linkcolor=blue, citecolor=blue, urlcolor=blue]{hyperref}

\usepackage{mathtools,amsmath, amsfonts, amssymb, latexsym, amsthm,
  amscd, cancel, enumerate, rotating, comment, appendix, ulem, 
  makecell, paralist,times,graphicx,pgf,tikz,tikz-cd, times,algorithm,
  algpseudocode,dirtytalk,siunitx}
\usepackage{extarrows}
\usepackage[margin=1.4in,letterpaper]{geometry} 
\usepackage[capitalise, noabbrev]{cleveref}
\usepackage[all]{xy}
\usepackage[english]{babel}
\usepackage[mathscr]{eucal}
\usepackage{xxcolor}
\usepackage{xcolor}
\usepackage{colortbl}
\usepackage{caption}
\usepackage{subcaption}
\usepackage[T1]{fontenc}      
\usepackage{ytableau}
\usetikzlibrary{shapes}
\usetikzlibrary{snakes}
\usetikzlibrary{arrows}
\usepackage{tikz-cd,leftindex,tensor}
\usetikzlibrary{matrix, decorations.pathmorphing}
\usetikzlibrary{calc}

\usetikzlibrary{positioning}

\usetikzlibrary{decorations.markings}
\tikzset{degil/.style={
            decoration={markings,
            mark= at position 0.5 with {
                  \node[transform shape] (tempnode) {$\backslash$};
                  }
              },
              postaction={decorate}
}
}

\usepackage{BOONDOX-uprscr}

\algnewcommand\algorithmicforeach{\textbf{for each}}
\algdef{S}[FOR]{ForEach}[1]{\algorithmicforeach\ #1\ \algorithmicdo}

\newif\ifdviwin
\numberwithin{equation}{section}

\theoremstyle{plain}
\newtheorem{theorem}{Theorem}[section]
\newtheorem{conjecture}[theorem]{Conjecture}

\newtheorem*{Main Theorem}{Main Theorem}

\newtheorem{proposition}[theorem]{Proposition}
\newtheorem{lemma}[theorem]{Lemma}
\newtheorem{corollary}[theorem]{Corollary}
\newtheorem{claim}[theorem]{Claim}
\newtheorem{problem}[theorem]{Problem}

{\Alph{theorem}}
{\Alph{theorem}}

\newtheorem{lthm}{Theorem}

\theoremstyle{definition}

\newtheorem{remark}[theorem]{Remark}
\newtheorem{definition}[theorem]{Definition}

\DeclareMathOperator{\TB}{TB}
\DeclareMathOperator{\tb}{tb}

\begin{document}

\title{On Legendrian Thurston-Bennequin-symmetrical graphs}
\author[Trung Chau]{Trung Chau}
\address{Chennai Mathematical Institute, India;\tiny{chauchitrung1996@gmail.com}}
\author[Tanushree Shah]{Tanushree Shah}
\address{Chennai Mathematical Institute, India; \tiny{tanushreebshah@gmail.com}}

 \thanks{\emph{Subjclass[2020]}: 57K33, 05C10}

\begin{abstract}
 This article reviews the development of Legendrian graph theory in the standard contact 3-sphere ($S^3, \xi_{std}$). We provide a generalized criterion under which the total Thurston-Bennequin invariant of a Legendrian graph (sum of tb of all cycles of the Legendrian graph) can be computed from the tb of its smaller cycles. We verify this criterion for graphs with up to 9 vertices and construct infinite families of examples where it holds. We also present examples demonstrating that each condition in the criterion is necessary. Notably, the graphs satisfying this criterion exhibit a high degree of symmetry.
\end{abstract}
\maketitle
\section{Introduction}

Legendrian graphs lie at the intersection of low-dimensional topology, contact geometry, and graph theory. Let $M$ be a smooth manifold of dimension $2n+1$. A contact structure on $M$ is a completely non-integrable tangent hyperplane distribution $\xi \subset TM$. Locally, $\xi$ can always be described as the kernel of a 1-form $\alpha$, known as a contact form. The non-integrability condition is encapsulated by the requirement that $\alpha \wedge (d\alpha)^n \neq 0 $
everywhere on $M$, which implies that this wedge product forms a non-vanishing volume form, orienting the manifold.
In this paper, we primarily restrict our attention to the standard contact 3-space, denoted $(\mathbb{R}^3, \xi_{std})$. The standard contact structure is typically defined by the kernel of the standard 1-form:
$\alpha_{std} = dz - y\,dx $

Within this rigid, twisting geometric framework, the submanifolds of interest are Legendrian curves. A Legendrian knot or link is a smooth, closed 1-dimensional submanifold $L \subset (M, \xi)$ that is everywhere tangent to the contact planes; that is, $T_x L \subset \xi_x$ for all $x \in L$. Legendrian knots have two classical invariants, the Thurston-Bennequin number ($tb$) and the rotation number ($rot$). These invariants measure, respectively, the twisting of the contact framing relative to the topological Seifert framing, and the winding of the tangent vector of the knot after trivializing the contact plane along the knot.

\subsection{Generalizing Knots to Spatial Graphs}

 Let $G$ be a finite 1-dimensional CW-complex. An embedding of $G$ into a contact 3-manifold (typically $\mathbb{R}^3$) is termed a \textit{Legendrian spatial graph} if its edges are smoothly embedded arcs everywhere tangent to the underlying contact distribution, and its vertices form well-defined branching points where multiple Legendrian arcs converge. 
To rigorously generalize classical framing invariants, such as the Thurston-Bennequin number ($tb$), one must upgrade the abstract topological graph to a ribbon graph. This structure tracks the torsional twisting of the bands if the graph were thickened into a surface. Consequently, the study of Legendrian graphs is not merely a generalization of knot theory, but the development of an essential mathematical framework for analyzing singular, stratified objects in constrained geometries.

Establishing this geometric foundation naturally leads to the ``Legendrian realization problem'': determining the specific structural constraints imposed upon an abstract topological graph when it is forced into a rigid contact embedding. The rigidity of the standard contact structure $\xi_{std}$ strictly limits how abstract graphs can be embedded in $\mathbb{R}^3$. By translating geometric optimization problems, such as maximizing the Thurston-Bennequin number into topological ones, researchers have uncovered unexpected  connections to structural graph minor theory \cite{baader2009}. In their case, the question of simultaneously maximizing the Thurston-Bennequin number across all cycles reduces to the absence of a finite graph minor. This naturally leads to a broader perspective, rather than asking whether all cycles can achieve extremal values, one may instead seek to understand how the Thurston-Bennequin invariants of different cycles are related to one another within a fixed embedding. In particular, the overlap of cycles along shared edges and vertices suggests that their writhe and cusp contributions are not independent, but satisfy systematic relations dictated by the combinatorics of the graph.

Motivated by this viewpoint, we develop a framework that captures such dependencies and translates them into explicit algebraic relations among Thurston-Bennequin invariants of cycles of different lengths.
\subsection{Statement of Main Results}
The main result of this paper establishes a structural relation between Thurston-Bennequin invariants of cycles of different lengths in a Legendrian embedding of a graph. We introduce a combinatorial symmetry condition, termed $(r,s)$-Thurston-Bennequin-symmetricity (or $(r,s)$-TB-symmetricity for short; for precise definition see Definition \ref{def:TB-symmetrical}), which encodes how edges and pairs of edges are distributed among $r$- and $s$-cycles, together with compatibility conditions governing their contributions to writhe and cusps. Our theorem shows that for any Legendrian embedding of a graph satisfying this condition, the total Thurston-Bennequin number over all $r$-cycles is proportional to that over all $s$-cycles, with a constant depending only on the underlying graph. In particular, this reduces the computation of Thurston-Bennequin invariants for larger cycles to those of smaller cycles, and provides a systematic method to detect constraints on Legendrian embeddings arising from the combinatorics of the graph. Precisely, we prove the following: 

\begin{lthm}[Theorem \ref{thm:TB-multiple-sum}]
    Let $\tilde{G}$ be a Legendrian embedding of a finite simple graph $G$ in $(\mathbb{R}^3,\zeta_{std})$. Let $s\geq 3$ be an integer such that $G$ has an $s$-cycle. If $G$ is TB-symmetrical, then \[
    \TB(\tilde{G})= \left(1+\sum_{r\geq 3, r\neq s} \rho_{r,s}(G) \right) \TB_s(\tilde{G}).\] 
\end{lthm}

TB-symmetrical graphs and a larger class, almost-TB-symmetrical graphs, are therefore of interests as their total Thurston-Bennequin number admits such a beautiful formula. We remark that it is previously known in \cite{MR3654493} that complete and complete bipartite graphs are TB-symmetrical. We recover their total Thurston-Bennequin number formulae in Theorem~\ref{thm:Kn-Kmn-TB}.

We aim to attack the following problem:
    \begin{problem}
    Characterize/Find all almost-TB-symmetrical and TB-symmetrical graphs.
\end{problem}
    While a full answer eludes us, we established many operations to obtain more (almost-)TB-symmetrical graphs. We give examples of highly symmetric graphs that one might expect to be TB-symmetrical but are not. A main result we obtained is that 2-arc transitive graphs are almost-TB-symmetrical, in which case we can relate the new invariants $\rho_{r,s}(G)$ with the cycle polynomial of $G$. For precise definition of $s$-arc transitive graphs, we refer to Section 5.

    \begin{lthm}[Theorem \ref{thm:2-arc}]
    Let $G$ be a $2$-arc transitive graph. Then $G$ is almost-TB-symmetrical. Moreover, if $r,s\geq 3$ are distinct integers such that $G$ has at least one $s$-cycle, then $\rho_{r,s}(G)=\frac{rc_r(G)}{sc_s(G)}$, where $c_r(G)$ and $c_s(G)$ are the numbers of $r$-cycles and $s$-cycles of $G$, respectively.
\end{lthm}
    
    We sum up some of our results in the following diagram of implications regarding graphs:
    
\begin{center}
         \begin{tikzcd}
    \text{4-arc transitive} \arrow[d,Rightarrow] \arrow[drr, Rightarrow, degil, "Proposition~\ref{prop:K222-Heawood-Petersen}"] && \\
    \text{3-arc transitive} \arrow[d,Rightarrow] && \text{TB-symmetrical} \arrow[d,Rightarrow] \arrow[dddll, Rightarrow, bend left=130,degil, "\text{\cite{MR3654493}}"]   \\
    \text{2-arc transitive}  \arrow[d,Rightarrow]  \arrow[rr,Rightarrow, "Theorem~\ref{thm:2-arc}"] &&\text{\small almost-TB-symmetrical} \arrow[u, Rightarrow, bend left, degil, "Proposition~\ref{prop:K222-Heawood-Petersen}"] \\
    \text{1-arc transitive} \arrow[d,Rightarrow] \arrow[urr, Rightarrow, degil, "Proposition~\ref{prop:K222-Heawood-Petersen}"] && \\
    \text{0-arc transitive}  &&
    \end{tikzcd}
    \end{center}

    We raise the following conjecture which states that odd graphs $O_n$ are TB-symmetrical, which we have resolved in the case $n=2,3$ (Proposition~\ref{prop:K222-Heawood-Petersen}).
    \begin{conjecture}\label{conj:On}
        The odd graph $O_n$ is TB-symmetrical for any $n\geq 2$.
    \end{conjecture}
    We also obtain the full list of all (almost-)TB-symmetrical graphs with at most 9 vertices (see Theorem~\ref{thm:9}).

\subsection{Structure of the Paper}
 Section 2 rigorously defines the standard contact structure, front and Legendrian projections, framing conventions, and the generalized Reidemeister moves for Legendrian graphs, the construction and computation of classical invariants, specifically extending the Thurston-Bennequin to spatial graphs equipped with ribbon structures. Section 3 details surveys the landscape of current realization theorems and upper bounds. Section 4 proves our main results regarding computation of total $TB$ of Legendrian $(r,s)$-Thurston-Bennequin-symmetrical graphs, and our recovery of the main results in \cite{MR3654493} in the case of complete and complete bipartite graphs. In section 5 we together with many graph operations that result in (almost-)TB-symmetrical graphs. In section 5 we  relate the concepts of (almost-)Thurston-Bennequin-symmetrical graphs with $s$-arc transitive graphs, together with the connection of the new invariants $\rho_{r,s}(G)$ with the cycle polynomial of $G$. Finally, in section 6, we provide a complete list of (almost-)TB-symmetrical graphs  with at most 9 vertices. 

\section*{Acknowledgements}

The authors are supported by the Infosys Foundation. We would like to thank Aryaman Maithani for many helpful conversations.  We have made extensive use of the computer algebra systems \texttt{Sage} \cite{sagemath}, and the package \texttt{nauty} \cite{nauty}; the use of these is gratefully acknowledged. 
 
\section{Preliminaries and Background}

This section formally defines the standard contact space, the front projections, and the extension of classical Legendrian invariants to graphs.

\subsection{The Standard Contact Space and Legendrian Embeddings}
We restrict our attention to the standard contact 3-space, denoted $(\mathbb{R}^3, \xi_{std})$. The standard contact structure $\xi_{std}$ is a completely non-integrable 2-plane distribution defined globally as the kernel of the standard contact 1-form:
$$ \alpha = dz - y\,dx $$
At each point $p = (x,y,z) \in \mathbb{R}^3$, the contact plane $\xi_p$ is spanned by the vectors $\partial_y$ and $\partial_x + y\partial_z$. Geometrically, as one moves along the $y$-axis, the contact planes remain flat (parallel to the $xy$-plane), but as one moves along the $x$-axis, the planes twist with a slope of $y$. This twisting is the fundamental source of rigidity in contact topology.

Let $G$ be a finite, abstract 1-dimensional CW-complex (a graph). A continuous map $f: G \hookrightarrow \mathbb{R}^3$ is an \textit{embedding} of a spatial graph. 
\begin{definition}

 An embedding $f: G \hookrightarrow (\mathbb{R}^3, \xi_{std})$ is a \textbf{Legendrian graph} if:
\begin{enumerate}
    \item The restriction of $f$ to the interior of any edge $e \subset G$ is a smooth embedding.
    \item For every point $p$ on the interior of an edge $e$, the tangent vector to the edge lies entirely within the contact plane at that point: $T_p(f(e)) \subset \xi_{std}|_p$. This implies that along any edge, the coordinate functions satisfy $dz = y\,dx$.
    \item At every vertex $v \in G$, the edges incident to $v$ meet smoothly and have distinct tangent directions within the 2-dimensional contact plane $\xi_{std}|_{f(v)}$.
\end{enumerate}
\end{definition}
Two Legendrian graphs are considered \textit{Legendrian isotopic} if there exists a smooth ambient isotopy of $\mathbb{R}^3$ that preserves the contact structure $\xi_{std}$ at each stage and carries one graph onto the other.

 The \textbf{front projection} is the mapping $\Pi_{xz}: \mathbb{R}^3 \to \mathbb{R}^2$ given by $\Pi_{xz}(x, y, z) = (x, z)$.

For a generic Legendrian graph $\Gamma$, its front projection $\Pi_{xz}(\Gamma)$ determines the entire 3-dimensional embedding. Because $\Gamma$ is Legendrian, it must satisfy $dz - y\,dx = 0$. Consequently, the missing $y$-coordinate can be completely recovered from the $xz$-projection via the slope of the curve:
$$ y = \frac{dz}{dx} $$
This recovery equation imposes three severe geometric restrictions on the front projection of any Legendrian graph:
\begin{enumerate}
    \item \textbf{No Vertical Tangencies:} Because $y$ must be finite, the front projection cannot have vertical tangent lines (where $dx = 0$). Instead, strands must turn around via s singularities, known as \textit{cusps} as shown in Fig. \ref{RLC}. Let $c(K)$ denote the total number of cusps (equivalently, the sum of left and right cusps).
    \item \textbf{Crossing Information is Fixed:} At any generic crossing in the $xz$-plane, the two strands share the same $x$ and $z$ coordinates. The strand with the smaller slope $dz/dx$ (more negative) corresponds to a smaller $y$-coordinate. In the standard orientation, the $y$-axis points "into" the page, meaning the strand with the smaller slope must always cross \textit{in front of} the strand with the larger slope. Thus, knotting information is not arbitrary; it is strictly dictated by the local geometry.
    \item \textbf{Vertex Transversality:} Because all edges meeting at a vertex $v$ must have distinct tangent directions in the contact plane $\xi_v$, their projections in the $xz$-plane must meet with pairwise distinct slopes. No two edges can be tangent at a vertex in the front projection.
\end{enumerate}

\begin{definition}

The \emph{writhe}, denoted $w(K)$, is the signed count of transverse double points (crossings) in the front projection. Each crossing is assigned a sign $+1$ or $-1$ for positive and negative crossing resp. as shown in Fig. \ref{PNC}. So $w(K)$ is the number of positive crossings minus the number of negative crossings.
\end{definition}
\begin{figure}
   \centering
    \includegraphics[width=5.5in]{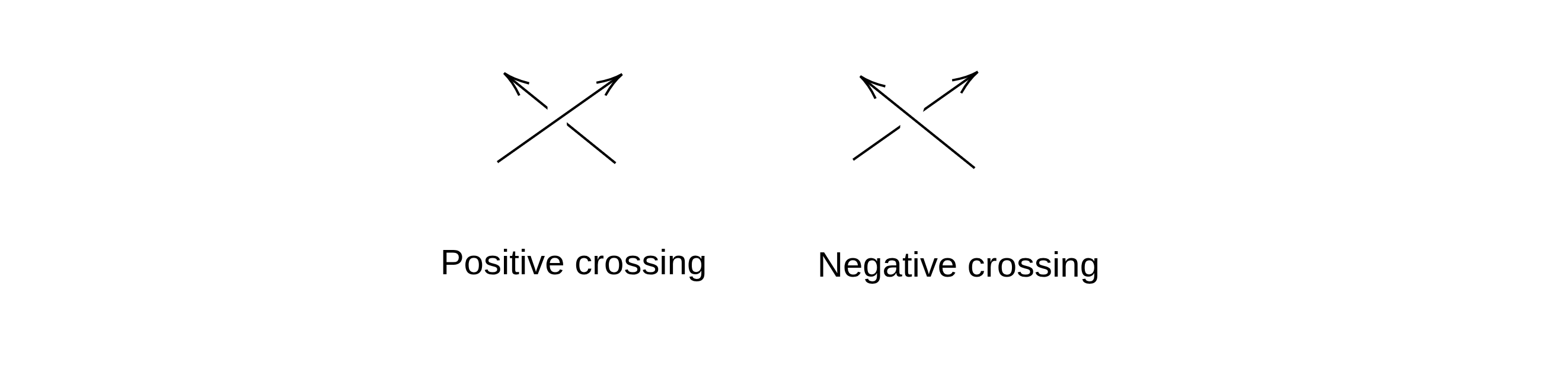}
     \caption{Positive and Negative crossings }
    \label{PNC}
\end{figure}

\begin{figure}
   \centering
    \includegraphics[width=5.5in]{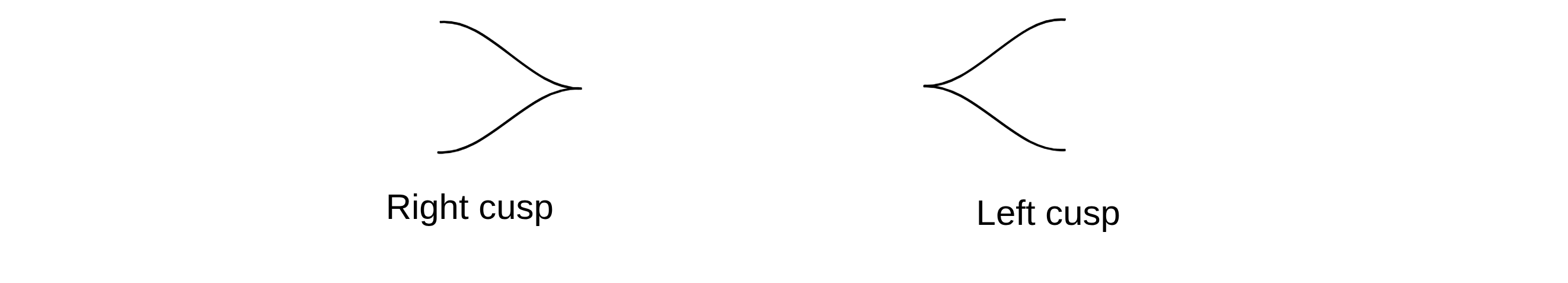}
     \caption{Right and Left cusp}
    \label{RLC}
\end{figure}
\subsubsection{Generalized Legendrian Reidemeister Moves}
Two Legendrian knots are isotopic if and only if their front projections are related by planar isotopies and the three classical Legendrian Reidemeister moves (R1, R2, R3) and three more Legendrian graph moves (R4, R5, R6) as shown in Fig. \ref{RM}. Together, these are called generalised Legendrian Reidemeister moves. We briefly explain the moves here. For more details, we refer interested readers to \cite{CD}.
 
\begin{itemize}
\item[Move IV (edge-vertex detour)]
{An edge incident to a vertex is locally deformed so that a small segment of the edge is pulled 
away from (or pushed toward) the vertex, creating or removing a pair of cusps near the vertex. 
This modifies the local geometry of the edge while keeping its endpoint fixed at the vertex and 
preserving the cyclic order at the vertex.}

    \item[Move V (vertex-edge slide)]
{An edge passing near a vertex is isotoped across the vertex, moving from one side to the other, 
without passing through it. The cyclic order of half-edges at the vertex is preserved, and the 
deformation respects the Legendrian conditions in the front projection.}
\item[Move VI (vertex crossing change).]
{Two edges incident to (or near) a vertex are allowed to switch their relative position via a local 
isotopy near the vertex. This changes the crossing between strands adjacent to the vertex while 
preserving the vertex structure and Legendrian conditions.}
\end{itemize}
\begin{figure}
   \centering
    \includegraphics[width=5.5in]{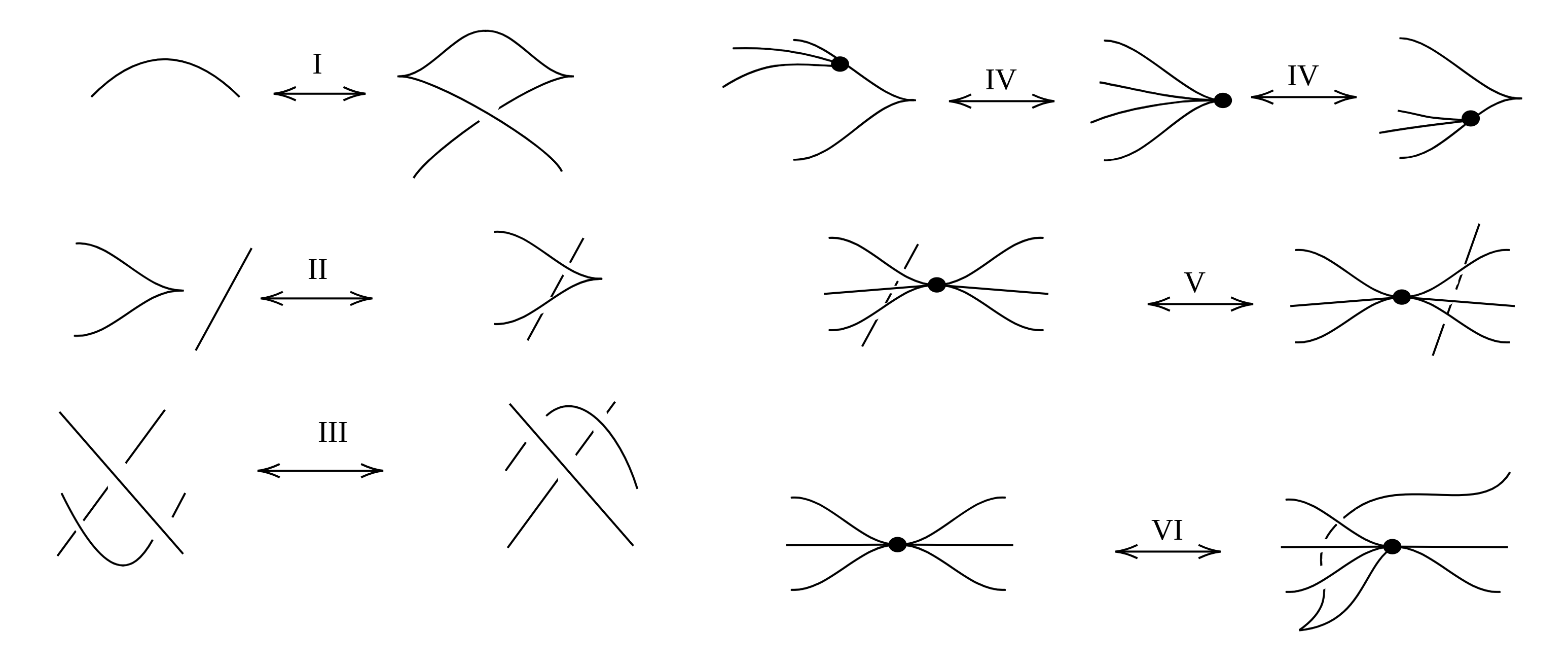}
     \caption{Generalized Legendrian Reidemeister Moves }
    \label{RM}
\end{figure}
\subsection{Thurston-Bennequin Invariant}

 For a Legendrian knot $K$, the \textbf{Thurston-Bennequin number}, $tb(K)$, measures the twisting of the contact framing relative to the Seifert framing. In the front projection, it is easily computed via the writhe $w(K)$ and the number of right cusps $c_R(K)$:
$$ tb(K) = w(K) - c_R(K) $$

For a Legendrian graph $\Gamma$, $tb$ cannot be uniquely defined without the ribbon structure, because a generic graph does not have a Seifert surface.

\subsubsection{Ribbon Structures and Vertex Rigidities}

A \textbf{ribbon graph} is an abstract graph $G$ together with a chosen cyclic ordering of the half-edges incident to each vertex. Topologically, this is equivalent to thickening the graph into an orientable surface with boundary, where the vertices become disks, and the edges become narrow bands attaching to the disks according to the cyclic order.

In contact geometry, a Legendrian graph inherently possesses a natural geometric ribbon structure. At any vertex $v$, the contact plane $\xi_v$ is 2-dimensional. The tangent vectors of the incident edges all lie in $\xi_v$. The orientation of the contact plane induces a natural cyclic ordering of these tangent vectors. 
However, when manipulating these graphs (specifically during vertex isotopies), tracking this ribbon structure is paramount. To compute classical invariants, we often "resolve" the vertices of the Legendrian graph by replacing them with a rigid, non-intersecting configuration of bands that preserves the topological ribbon structure.

First, we resolve every vertex $v$ of $\Gamma$ according to its cyclic ribbon structure. This involves replacing the vertex with a set of small, non-intersecting arcs that connect the edges in a way that traces the boundary of the thickened ribbon surface. This resolution process effectively converts the Legendrian graph into a multicomponent Legendrian link, denoted $L(\Gamma, R)$. The Thurston-Bennequin number of the graph is then defined as the $tb$ of this resolved link, adjusted by the internal framing of the vertices:
$$ tb(\Gamma, R) = tb(L(\Gamma, R)) + \sum_{v \in V} \text{tw}(v) $$
where $\text{tw}(v)$ is a correction term (often multiples of $1/2$) accounting for the half-twists introduced by projecting the 2-dimensional contact plane to the $xz$-plane at the vertex.

With these conventions, the Thurston-Bennequin invariant is given by
\[
tb(\Gamma) = w(\Gamma) - \frac{1}{2} c(\Gamma).
\]

Crossings contribute $\pm 1$ to the writhe depending on their sign, while each cusp contributes $-\tfrac{1}{2}$ to the Thurston--Bennequin invariant.

To capture the overall geometric behavior of a Legendrian graph, invariants can be defined over the entire set of cycles within the graph. 

\begin{definition}
   For a Legendrian graph $G$ and an integer $r\geq 3$, the \emph{$r$-Thurston-Bennequin number of $G$}, denoted $\TB_r(G)$, is defined as the sum of the Thurston-Bennequin numbers over all $r$-cycles of $G$, where each cycle is considered as a knot.  

   For a Legendrian graph $G$, the total Thurston-Bennequin number of $G$, denoted $\TB(G)$, is defined as the sum of the Thurston-Bennequin numbers over all cycles of $G$, where each cycle is considered as a knot. In particular, we have $\TB(G)=\sum_{r\geq 3} \TB_r(G)$ 
\end{definition}

\begin{figure}
   \centering
    \includegraphics[width=5in]{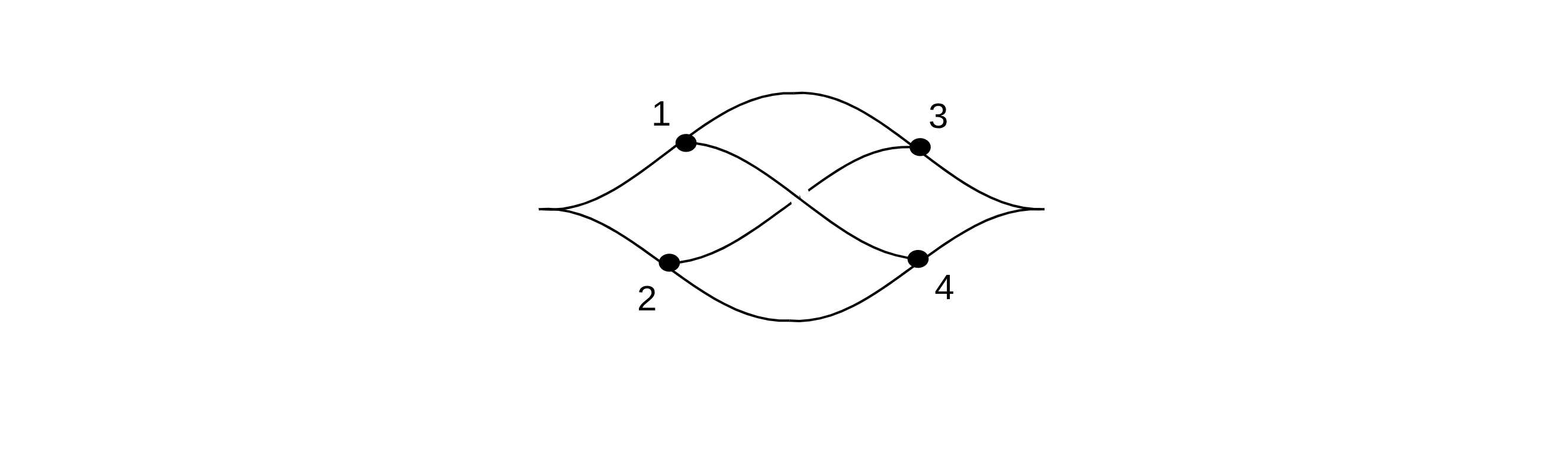}
     \caption{Legendrian K4 graph}
    \label{K4}
\end{figure}
Take the Legendrian embedding of $K_4$ as shown in Fig. \ref{K4}. It has four 3-cycles, namely (123, 124, 134, 234)  and three 4-cycles, namely (1234, 1243, 1324). The cycle (1243) has $\tb=2$. Every other cycle has $tb=-1$.

\section{Minimal embeddings}
One of the central questions in the field typically revolves around the ``Legendrian Realization Problem'' (which topological graphs can be embedded with specific contact constraints), the boundedness of classical invariants, and the algebraic classification of these embeddings. Below is a comprehensive survey of the foundational theorems that define the current state of the field.

\subsection{Legendrian Approximation and Boundedness of Invariants}
Before classifying specific embeddings, one must establish that Legendrian embeddings of arbitrary spatial graphs exist and that their invariants are well-behaved. O'Donnol established the foundational approximation and bounding theorems for spatial networks, directly generalizing the classical results for Legendrian knots.

\textbf{Theorem 3.1 (Legendrian Approximation, O'Donnol \cite{odonnol2015}).} \textit{Let $G$ be a finite topological spatial graph embedded in $(\mathbb{R}^3, \xi_{std})$. For any $\epsilon > 0$, there exists a Legendrian spatial graph $\Gamma$ that is $C^0$-close to $G$ (specifically, within an $\epsilon$-neighborhood of $G$). Furthermore, if $G$ is already Legendrian on a closed subgraph, the approximation can be made relative to that subgraph.}

While any graph can be approximated by a Legendrian graph, the twisting of the contact structure imposes strict limits on the framing of the embedding.

\textbf{Theorem 3.2 (Maximal Thurston-Bennequin Bound, O'Donnol \cite{odonnol2015}).} \textit{Let $\mathcal{K}$ be a fixed topological isotopy class of a spatial graph. The set of Thurston-Bennequin numbers, $\{tb(\Gamma, R)\}$, for all Legendrian graphs $\Gamma \in \mathcal{K}$ equipped with a fixed ribbon structure $R$, is bounded from above. Consequently, there exists a maximal Thurston-Bennequin number, denoted $\overline{tb}(\mathcal{K}, R)$.}

In particular, the Thurston--Bennequin invariant of a Legendrian graph (with respect to a fixed ribbon structure) is always \emph{bounded above}, reflecting the tightness of the ambient contact structure and placing strong constraints on admissible embeddings.

\subsection{Minimal embeddings}
Motivated by these bounds, one is led to consider extremal embeddings.

\textbf{Definition 3.4.} A Legendrian embedding $\Gamma$ of a spatial graph $G$ is called \emph{minimal} (with respect to a fixed ribbon structure $R$) if it realizes the maximal Thurston--Bennequin number, i.e.,
\[
tb(\Gamma, R) = \overline{tb}(\mathcal{K}, R),
\]
where $\mathcal{K}$ is the topological isotopy class of $G$.

More generally, one may consider minimality conditions imposed cycle-wise, requiring that each cycle in the graph realizes its maximal Thurston--Bennequin number within its knot type.

\subsubsection{Combinatorial Rigidity and the $K_4$ Minor Obstruction}
One of the most striking results in contact topology connects the geometric rigidity of the standard contact structure to the purely combinatorial realm of the Robertson--Seymour graph minor theorem.

When attempting to realize an abstract graph as a Legendrian graph, one naturally asks if every cycle within the graph can simultaneously achieve its maximal possible Thurston--Bennequin number (for an unknot, this maximum is $tb = -1$). Baader and Ishikawa proved that this geometric optimization is obstructed by a specific finite subgraph.

\textbf{Theorem 3.3 (The $K_4$ Minor Obstruction, Baader \& Ishikawa \cite{baader2009}).} \textit{Let $G$ be a bipartite abstract graph. $G$ can be realized as a Legendrian spatial graph in $(\mathbb{R}^3, \xi_{std})$ such that every cycle in $G$ is a Legendrian unknot with $tb = -1$ if and only if $G$ does not contain the complete graph on four vertices, $K_4$, as a graph minor.}

The classical bounds on $tb$ for knots impose strict global limitations on the cycles within a graph. O'Donnol and Pavelescu proved that a graph admits a Legendrian embedding where all its cycles are trivial unknots if and only if it does not contain $K_4$ as a minor \cite{Tan}. For $K_4$ itself, there does not exist a Legendrian embedding where all cycles realize maximal Thurston-Bennequin numbers consisting solely of odd numbers; however, one can construct an infinite family of Legendrian embeddings of $K_4$ where all cycles but one realize an odd maximal $tb$, and the final cycle realizes a maximal $tb$ equal to zero \cite{Tan}.

\medskip

These results indicate that the problem of determining minimal embeddings is subtle and governed by both local geometric constraints and global combinatorial obstructions. In particular, techniques for computing total Thurston-Bennequin numbers and understanding their behavior under graph operations can provide effective tools for obstructing the existence of minimal embeddings. We expect that the methods developed in the next section, concerning total $tb$ computations and related techniques, will be useful in detecting such obstructions.

\section{The Thurston-Bennequin number of a Legendrian $(r,s)$-Thurston-Bennequin-symmetrical graph}\label{sec:operations}

\begin{definition}\label{def:TB-symmetrical}
    Let $G$ be a finite simple graph and $r,s\geq 3$ be two distinct integers. If there exists a nonnegative number $\rho$ such that the following holds:
    \begin{enumerate}
        \item each edge $e$ of $G$ appears in $f_e$ $s$-cycles and $\rho f_e$ $r$-cycles;
        \item each pair of adjacent edges $p=(e,e')$ of $G$  appears in $g_p$ $s$-cycles and $\rho g_p$ $r$-cycles,
    \end{enumerate}
    then $G$ is called  \emph{almost-$(r,s)$-Thurston-Bennequin-symmetrical} (or almost-$(r,s)$-TB-symmetrical for short), and denote $\rho$ by $\rho_{r,s}(G)$. We adopt the convention $\rho_{r,s}(G)=0$ if $G$ has neither $r$-cycles nor $s$-cycles. Moreover, if $\rho$ also satisfies the following condition:
    \begin{enumerate}
        \item[(3)] each pair of non-adjacent edges $p=(e,e')=(\{a,b\},\{c,d\})$ appears in $u_{p}$ $s$-cycles and $v_{p}$ $r$-cycles with orientation $\cdots ab\cdots cd\cdots$, and $h_{p}$ $s$-cycles and $k_{p}$ $r$-cycles with orientation $\cdots ab\cdots dc\cdots$, with $v_p-k_p=\rho(u_p-h_p)$, 
    \end{enumerate}
    then $G$ is called \emph{$(r,s)$-Thurston-Bennequin-symmetrical} (or $(r,s)$-TB-symmetrical for short).
    
    We call $G$ \emph{almost-TB-symmetrical} (resp, TB-symmetrical) if it is either almost-$(r,s)$-TB-symmetrical (resp, $(r,s)$-TB-symmetrical) or almost-$(s,r)$-TB-symmetrical (resp, $(s,r)$-TB-symmetrical)  for any distinct integers $r, s\geq 3$.
\end{definition}

We remark the following trivial case of (almost-)$(r,s)$-symmetrical graphs.

\begin{lemma}\label{lem:same-size-cycles}
    Let $G$ be a graph and $r\geq 3$ a positive integer. Then the following are equivalent:
    \begin{enumerate}
        \item $G$ has no $r$-cycle;
        \item $G$ is (almost-)$(r,s)$-TB-symmetrical for any $s\geq 3$ where $s\neq r$, with $\rho_{r,s}(G)=0$;
        \item $G$ is (almost-)$(r,s)$-TB-symmetrical for some $s\geq 3$ where $s\neq r$, with $\rho_{r,s}(G)=0$.
    \end{enumerate}
    In particular, if all cycles of $G$ (if any) are of the same size, then $G$ is TB-symmetrical.
\end{lemma}
\begin{proof}
    \emph{(2) $\implies$ (3):} Obvious.

    \emph{(3) $\implies$ (1):} Suppose that $G$ is (almost-)$(r,s)$-TB-symmetrical for some $s\geq 3$ where $s\neq r$, with $\rho_{r,s}(G)=0$. By (1) of Definition~\ref{def:TB-symmetrical}, any edge $e$ of $G$ appears in $\rho_{r,s}(G) f_e=0$ $r$-cycles of $G$, for some integer $f_e$. In other words, $G$ has no $r$-cycle, as desired.

    \emph{(1) $\implies$ (2):} If $G$ has no $r$-cycle, (2) follows from definition with $\rho_{r,s}(G)=0$. 
\end{proof}

In particular, checking (almost-)$(r,s)$-TB-symmetricity is only non-trivial if $\rho_{r,s}(G)$ is nonzero, in which case we have the following.

\begin{lemma}\label{lem:reverse-r-s}
    Let $G$ be a graph and $r,s\geq 3$ two distinct integers. Then $G$ is (almost-)$(r,s)$-TB-symmetrical with $\rho_{r,s}(G)\neq 0$ if and only if $G$ is (almost-)$(r,s)$-TB-symmetrical with $\rho_{s,r}(G)\neq 0$. Moreover, in this case, we have $\rho_{s,r}(G)=1/\rho_{r,s}(G)$.
\end{lemma}
\begin{proof}
    The result follows from verifying Definition~\ref{def:TB-symmetrical}.
\end{proof}

The idea is that when $G$ is TB-symmetrical, we obtain consequences for the total Thurston-Bennequin number of any Legendrian embedding of $G$.    

\begin{theorem}\label{thm:TB-multiple}
    Let $\tilde{G}$ be a Legendrian embedding of a finite simple graph $G$
in $(\mathbb{R}^3,\zeta_{std})$. Let $r, s\geq 3$ be two distinct positive integers. If $G$ is $(r,s)$-TB-symmetrical, then $\TB_r(\tilde{G})= \rho_{r,s}(G) \TB_s(\tilde{G})$. 
\end{theorem}

\begin{proof}
    Set $\rho_{r,s}(G)=\rho$. Items (1)-(3) in Definition~\ref{def:TB-symmetrical} give
    \[
    \sum_{\gamma\in \Gamma_r} w(\gamma) = \rho \sum_{\gamma\in \Gamma_s} w(\gamma).
    \]
    On the other hand, cusps occur either at a vertex, that is, at each pair of adjacent edges, or along one edge. Items (1) and (2) in Definition~\ref{def:TB-symmetrical} then give
    \[
    \sum_{\gamma\in \Gamma_r} c(\gamma) = \rho \sum_{\gamma\in \Gamma_s} c(\gamma).
    \]
    The result then follows.
\end{proof}

\begin{theorem}\label{thm:TB-multiple-sum}
    Let $\tilde{G}$ be a Legendrian embedding of a finite simple graph $G$ in $(\mathbb{R}^3,\zeta_{std})$. Let $s\geq 3$ be an integer such that $G$ has an $s$-cycle. If $G$ is TB-symmetrical, then \[
    \TB(\tilde{G})= \left(1+\sum_{r\geq 3, r\neq s} \rho_{r,s}(G) \right) \TB_s(\tilde{G}).\] 
\end{theorem}

\begin{proof}
    Let $r\geq 3$ be an integer such that $r\neq s$. We will show that $G$ is $(r,s)$-TB-symmetrical, and thus the result would follow from Theorem~\ref{thm:TB-multiple}. 

    If $G$ has no $r$-cycle, then by Lemma~\ref{lem:same-size-cycles}, $G$ is $(r,s)$-TB-symmetrical, as desired.

    Now we can assume that $G$ has an $r$-cycle. Since $G$ is TB-symmetrical, $G$ is either $(r,s)$-TB-symmetrical or $(s,r)$-TB-symmetrical. If the former holds, then we are done, and thus we can now assume that the latter holds, i.e., that $G$ is $(s,r)$-TB-symmetrical. By Lemma~\ref{lem:same-size-cycles}, since $G$ has an $s$-cycle by our hypothesis, we have $\rho_{s,r}(G)\neq 0$. By Lemma~\ref{lem:reverse-r-s},  $G$ is $(r,s)$-TB-symmetrical, as~desired.
\end{proof}

Our next goal is to recover the main results in \cite{MR3654493}. Recall that Next we provide some TB-symmetrical graphs. Recall that
\begin{itemize}
    \item a graph $G$ is called \emph{complete}, in which case $G$ is denoted by $K_{|V(G)|}$, if
    \[
    E(G)=\{\{v,v'\}\mid v,v'\in V(G) \text{ where } v\neq v'\};
    \]
    \item a graph $G$ is called \emph{complete bipartite} if $V(G)=A\sqcup B$ such that 
    \[
    E(G)=\{\{v,v'\}\mid v\in A, v'\in B\},
    \]
    in which case $G$ is denoted by $K_{|A|,|B|}$.
\end{itemize}

In \cite{MR3654493}, the authors essentially proved Theorem~\ref{thm:TB-multiple} in the case of a complete graph or a complete bipartite graph. We restate their result (with omitted proofs) in terms of our new terminology.

\begin{theorem}[\protect{\cite[Proofs of Theorems 2.2 and 3.1]{MR3654493}}]\label{thm:Kn-Kmn}
    For any integers $m,n\geq 1$, the graphs $K_n$ and $K_{m,n}$ are TB-symmetrical.
\end{theorem}

We remark that in the case of these graphs, the authors in \cite{MR3654493} also found $\rho_{r,s}(G)$ for any $r,s$. We will compute these invariants in a more general settings, and thus recover their results in full. Recall that a graph $G$ is called \emph{edge-transitive} if for any two edges $e_1$ and $e_2$ of $G$, there exists an automorphism of $G$ that maps $e_1$ to $e_2$. For a graph $G$ and a positive integer $s\geq 3$, let $c_s(G)$ denote the number of $s$-cycles in $G$. The \emph{cycle polynomial} of $G$ is defined to be
\[
C_G(t)=\sum_{s\geq 3} c_s(G) t^s.
\]

\begin{theorem}\label{thm:cycle-polynomial}
    Let $G$ be an edge-transitive and (almost-)TB-symmetrical graph. If $r,s\geq 3$ are distinct integers such that $G$ has at least one $s$-cycle, then $\rho_{r,s}(G)=\frac{rc_r(G)}{sc_s(G)}$.
\end{theorem}

\begin{proof}
    Since $G$ is edge-transitive, every edge of $G$ is in the same number of $r$-cycles, for each $r\geq 3$. Fix an integer $r\geq 3$, and let $\rho_r(G)$ be the number of $r$-cycles containing a fixed edge, which is independent of the edge being considered (since again, $G$ is edge-transitive). Consider all $c_r(G)$ $r$-cycles of $G$. We shall count the number of edges in all of these cycles. One straightforward answer is $rc_r$, as one $r$-cycle contains exactly $r$ edges. On the other hand, since any edge appears in exactly $\rho_r(G)$ $r$-cycles, the total number of edges must be $\rho_e|E(G)|$. Thus we have
    \[
    rc_r(G)=\rho_r(G)|E(G)| \implies \rho_r(G)= \frac{rc_r(G)}{|E(G)|}.
    \]
    Now consider two distinct integers $r,s\geq 3$ such that $G$ has at least one $s$-cycle. This condition guarantees that $c_s(G)\neq 0$. We then have
    \[
    \rho_{r,s}(G)=\frac{\rho_r(G)}{\rho_s(G)} = \frac{rc_r(G)}{sc_s(G)},
    \]
    as desired.
\end{proof}

We are now ready to recover the main results in \cite{MR3654493}. Remark that $K_n$ and $K_{m,n}$ are evidently edge-transitive by definition (see also an equivalent definition in \cite{edge-transitive}).

\begin{theorem}[\protect{\cite[Theorems 2.2 and 3.1]{MR3654493}}]\label{thm:Kn-Kmn-TB}
    Let $\tilde{G}$ be a Legendrian embedding of a finite simple graph $G$ in $(\mathbb{R}^3,\zeta_{std})$. If $G=K_n$, then
    \[
    \TB(\tilde{G})= \left(\sum_{r= 3}^n \frac{(n-3)!}{(n-r)!} \right) \TB_3(\tilde{G}).\] 
    If $G=K_{m,n}$, then 
    \[
    \TB(\tilde{G})= \left(\sum_{r= 2}^{\min \{m,n\}} \frac{(m-2)!(n-2)!}{(m-r)!(n-r)!} \right) \TB_4(\tilde{G}).\]
\end{theorem}

\begin{theorem}
    for any two distinct integers $3\leq s<r\leq n$
    and
    \[
    \rho_{2r,2s}(K_{m,n})= \frac{(n-s)!(m-s)!}{(n-r)!(m-r)!}
    \]
    for any two distinct integers $2\leq s<r\leq \min \{m,n\}$.
\end{theorem}
\begin{proof}
    Assume $G=K_n$. By Lemma~\ref{thm:TB-multiple-sum}, it suffices to show that $\rho_{r,3}(K_n)= \frac{(n-3)!}{(n-r)!}$ for any $4\leq r\leq n$.  Recall that the cycle polynomial of $K_n$ (see e.g., \cite{cycle})  is
\[
C_{K_n}(t) = \frac{1}{2}\sum_{s=3}^n \binom{n}{s}(s-1)! t^s.
\]
    By Theorems~\ref{thm:Kn-Kmn} and \ref{thm:cycle-polynomial}, we then have
\[
\rho_{r,3}(K_n) = \frac{rc_r(K_n)}{3c_3(K_n)} = \frac{r\frac{1}{2} \binom{n}{r}(r-1)! }{3\frac{1}{2}\binom{n}{3}(3-1)!} = \frac{(n-3)!}{(n-r)!},
\]
    as desired.

    Now assume $G=K_{m,n}$. Note that $G$ only has even cycles. By Lemma~\ref{thm:TB-multiple-sum}, it suffices to show that $\rho_{2r,4}(K_{m,n})= \frac{(m-2)!(n-2)!}{(m-r)!(n-r)!}$ for any $2\leq r \leq \min\{m,n\}$. We have  the following claim.

    \begin{claim}\label{clm:Kmn}
        The number of $2r$-cycles in $K_{m,n}$ is 
        \[
        c_{2r}(K_{m,n})= \frac{1}{2}\binom{m}{r}\binom{n}{r}r!(r-1)!.
        \]
    \end{claim}
    \begin{proof}[Proof of Claim~\ref{clm:Kmn}]
        Let $V(K_{m,n})=A\sqcup B$ where there is no edge among vertices of $A$, among vertices of $B$, and $|A|=m$, $|B|=n$. A $2r$-cycle of $K_{m,n}$ is formed by $r$ vertices in $A$ and $r$ vertices in $B$, put in an alternating order. There are $\binom{m}{r}\binom{n}{r}$ ways to pick these vertices. Let $v_1,\dots, v_r\in A$ and $u_1,\dots, u_r\in B$ be the vertices we picked to form a $2r$-cycle. There are $\binom{r}{2}$ ways to pick the two vertices in $B$ that are adjacent of $v_1$, which we call $u_1'$ and $u_2'$. Now there are exactly $r-1$ options for the vertex that is adjacent to $u_2'$ (that is not $v_1$), which we shall call $v_2'$. We continue this until we complete the $2r$-cycle. In total, the total number of $2r$-cycles are
        \[
        \binom{m}{r}\binom{n}{r} \binom{r}{2}(r-1)(r-2)(r-2)\times \cdots \times 1\times1 = \frac{1}{2}\binom{m}{r}\binom{n}{r}r!(r-1)!,
        \]
        as desired.
    \end{proof}
    By Theorems~\ref{thm:Kn-Kmn} and \ref{thm:cycle-polynomial}, we then have
\[
\rho_{2r,4}(K_{m,n}) = \frac{2rc_{2r}(K_n)}{4c_4(K_{m,n})} = \frac{2r\frac{1}{2}\binom{m}{r} \binom{n}{r}r!(r-1)! }{4\frac{1}{2}\binom{m}{2}\binom{n}{2}2!1!} = \frac{(m-2)!(n-2)!}{(m-r)!(n-r)!},
\]
    as desired.
\end{proof}

\section{More (almost)-TB-symmetrical graphs, and arc transitivity}

In this section, we relate the concepts of (almost-)TB-symmetricity and arc transitivity, a terminology that measures the ``symmetricity" of a graph.

In graph theory, graphs with symmetricity are of great interests, and often manifest themselves in many different ways. We recall one such concept from \cite{MR1829620}. Let $s$ be a positive integer. An \emph{$s$-arc} in a graph is a sequence of vertices $v_0,\dots, v_s$ such that the consecutive vertices are adjacent, and that $v_{i-1}\neq v_{i+1}$ for any $0<i<s$. A graph $G$ is called \emph{$s$-arc transitive} if for any two $s$-arcs $v_0,\dots, v_s$ and $u_0,\dots, u_s$, there exists an automorphism of $G$
\begin{align*}
    f\colon V(G)&\to V(G)\\
    v_i&\mapsto u_i \text{ for each } 0\leq i\leq s.
\end{align*}
In other words, $G$ is $s$-arc transitive if its  automorphism group is transitive on $s$-arcs. It is straightforward that $s$-arc transitive graphs are also $(s-1)$-arc transitive, for any $s\geq 1$.
A $0$-arc transitive graph is also known as a \emph{vertex-transitive} graph, while a $1$-arc transitive graph is better known as a \emph{symmetric} graph. We remark that $1$-arc transitive graphs are edge-transitive, since a 1-arc is nothing but an edge equipped with a direction.

Our main result is that $2$-arc transitive graphs are almost-TB-symmetrical, in which case the constants $\rho_{r,s}$ can be computed using the graph's cycle polynomial.

\begin{theorem}\label{thm:2-arc}
    Let $G$ be a $2$-arc transitive graph. Then $G$ is almost-TB-symmetrical. Moreover, if $r,s\geq 3$ are distinct integers such that $G$ has at least one $s$-cycle, then $\rho_{r,s}(G)=\frac{rc_r(G)}{sc_s(G)}$.
\end{theorem}

\begin{proof}
    The second statement follows from Theorem~\ref{thm:cycle-polynomial} and the fact that 2-arc transitivity implies edge-transitivity. It remains to show the first statement.

    Since $G$ is $2$-arc transitive, it is $1$-arc transitve. In particular, for any two edges ($1$-arcs) of $G$, there exists an automorphism of $G$ that maps one to another. Thus every edge of $G$ is in the same number of $r$-cycles, for each $r\geq 3$, and thus condition (1) of Definition~\ref{def:TB-symmetrical} holds. On the other hand, any  pair of adjacent edges, given a direction, is a $2$-arc. Since $G$ is $2$-arc transitive, every pair of adjacent edges is in the same number of $r$-cycles, for each $r\geq 3$, and thus condition (2) of Definition~\ref{def:TB-symmetrical} holds. Therefore, $G$ is almost-TB-symmetrical, as desired.
\end{proof}

\begin{remark}\label{rem:code}
    Before moving on, we discuss the Sage code we write (\cite{code}) to verify whether a graph is almost-TB-symmetrical or TB-symmetrical. The process to verify almost-TB-symmetricity is taken directly from definition. It is harder to verify TB-symmetricity, as condition (3) in Definition~\ref{def:TB-symmetrical} requires one to be able to recognize the directions of a pair of edges in a given cycle on the machine. In our code to verify TB-symmetricity, we assume that $G$ is almost-TB-symmetrical to begin with, and that $0$ and $1$ are vertices of $G$ that form an edge of $G$ such that they are the first-named vertices. This guarantees that a cycle of $G$ obtained from the function \texttt{all\_simple\_cycles}, if it contains the edge 01, always receives an output as a sequence starting with 0 and 1, in that order. We also further assume that the input $G$ is 1-arc transitive, so that when verifying condition (3) in Definition~\ref{def:TB-symmetrical}, it suffices to assume the first edge is 01.
\end{remark}

We recall the classical complete tripartite graph $K_{2,2,2}$, the Heawood graph, the Petersen graph, and the cube graph $Q_3$, pictured below.

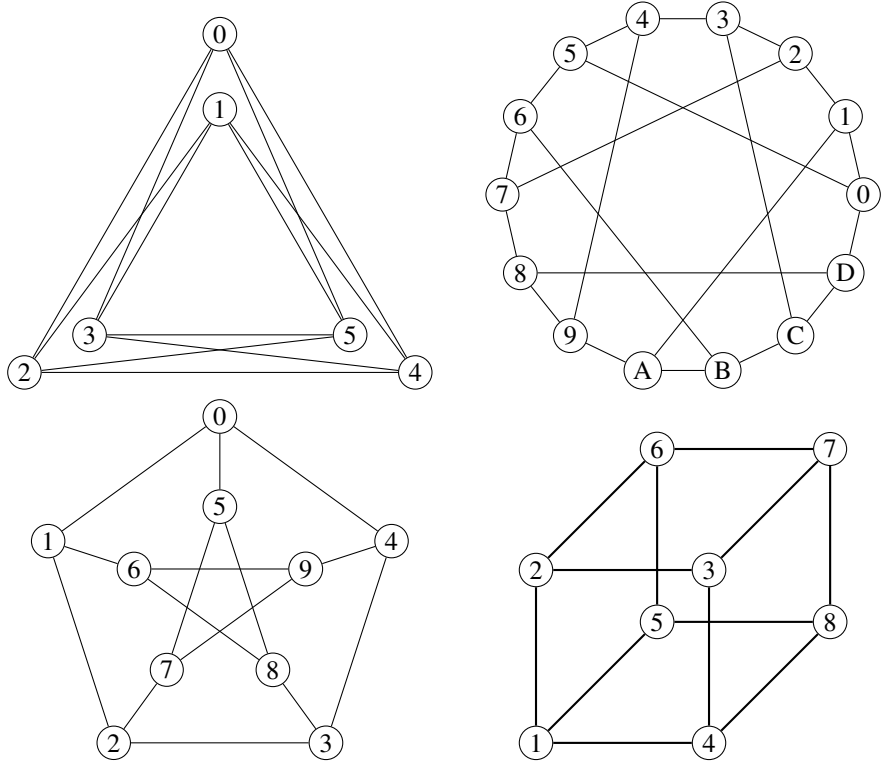
\begin{figure}[h!]
\centering
\begin{tabular}{ccc}
\begin{tikzpicture}[scale=2, every node/.style={circle,draw,fill=white,inner sep=1.5pt}]

\node (A1) at (90:1.5) {0};
\node (A2) at (90:1.0) {1};

\node (B1) at (210:1.5) {2};
\node (B2) at (210:1.0) {3};

\node (C1) at (330:1.5) {4};
\node (C2) at (330:1.0) {5};

\draw (A1)--(B1);
\draw (A1)--(B2);
\draw (A2)--(B1);
\draw (A2)--(B2);

\draw (B1)--(C1);
\draw (B1)--(C2);
\draw (B2)--(C1);
\draw (B2)--(C2);

\draw (A1)--(C1);
\draw (A1)--(C2);
\draw (A2)--(C1);
\draw (A2)--(C2);

\end{tikzpicture}
&&
\begin{tikzpicture}[scale=0.8, every node/.style={circle,draw,fill=white,inner sep=1.5pt}]
  \node (v0) at (0:3) {0};
  \node (v1) at (360/14:3) {1};
  \node (v2) at (2*360/14:3) {2};
  \node (v3) at (3*360/14:3) {3};
  \node (v4) at (4*360/14:3) {4};
  \node (v5) at (5*360/14:3) {5};
  \node (v6) at (6*360/14:3) {6};
  \node (v7) at (7*360/14:3) {7};
  \node (v8) at (8*360/14:3) {8};
  \node (v9) at (9*360/14:3) {9};
  \node (vA) at (10*360/14:3) {A};
  \node (vB) at (11*360/14:3) {B};
  \node (vC) at (12*360/14:3) {C};
  \node (vD) at (13*360/14:3) {D};

  \draw (v0) -- (v1);
  \draw (v1) -- (v2);
  \draw (v2) -- (v3);
  \draw (v3) -- (v4);
  \draw (v4) -- (v5);
  \draw (v5) -- (v6);
  \draw (v6) -- (v7);
  \draw (v7) -- (v8);
  \draw (v8) -- (v9);
  \draw (v9) -- (vA);
  \draw (vA) -- (vB);
  \draw (vB) -- (vC);
  \draw (vC) -- (vD);
  \draw (vD) -- (v0);

  \draw (v0) -- (v5);
  \draw (v2) -- (v7);
  \draw (v3) -- (vC);
  \draw (v1) -- (vA);
  \draw (v4) -- (v9);
  \draw (v6) -- (vB);
  \draw (v8) -- (vD);
\end{tikzpicture}
\\
\begin{tikzpicture}[scale=1.2, every node/.style={circle, draw, fill=white, inner sep=1.5pt}]

\node (v0) at (90:2) {0};
\node (v1) at (162:2) {1};
\node (v2) at (234:2) {2};
\node (v3) at (306:2) {3};
\node (v4) at (18:2) {4};

\node (u0) at (90:1) {5};
\node (u1) at (162:1) {6};
\node (u2) at (234:1) {7};
\node (u3) at (306:1) {8};
\node (u4) at (18:1) {9};

\draw (v0) -- (v1);
\draw (v1) -- (v2);
\draw (v2) -- (v3);
\draw (v3) -- (v4);
\draw (v4) -- (v0);

\draw (u0) -- (u2);
\draw (u2) -- (u4);
\draw (u4) -- (u1);
\draw (u1) -- (u3);
\draw (u3) -- (u0);

\draw (v0) -- (u0);
\draw (v1) -- (u1);
\draw (v2) -- (u2);
\draw (v3) -- (u3);
\draw (v4) -- (u4);

\end{tikzpicture}
&&
\begin{tikzpicture}[scale=2.3, every node/.style={circle, draw, fill=white, inner sep=1.5pt}]
  
  \node (v0) at (0,0) {1};
  \node (v1) at (0,1) {2};
  \node (v2) at (1,1) {3};
  \node (v3) at (1,0) {4};
  
  \node (v4) at (0.7,0.7) {5};
  \node (v5) at (0.7,1.7) {6};
  \node (v6) at (1.7,1.7) {7};
  \node (v7) at (1.7,0.7) {8};
  
  \draw[thick] (v0) -- (v1) -- (v2) -- (v3) -- cycle;
  
  \draw[thick] (v4) -- (v5) -- (v6) -- (v7) -- cycle;
  
  \draw[thick] (v0) -- (v4);
  \draw[thick] (v1) -- (v5);
  \draw[thick] (v2) -- (v6);
  \draw[thick] (v3) -- (v7);
  \draw[thick] (v4) -- (v7);
  \draw[thick] (v3) -- (v0);
  
\end{tikzpicture}
\end{tabular}
\caption{A complete tripartite graph $K_{2,2,2}$, a Heawood graph, a Petersen graph, and a cube graph $Q_3$, from left to right, from top to bottom, respectively.}
\end{figure}

We then have the following, verified with \cite{code}. Remark that  $K_{2,2,2}$ is known to be $1$-arc transitive graph (see e.g., \cite{arc-transitive}), the Heawood graph is known to be 4-arc transitive (see e.g., \cite{MR1327328}), and the Petersen graph is known to be 3-arc transitive (\cite{MR1373683}). On the other hand, $K_{2,3}$ is not 0-arc transitive (vertex-transitive) since it has two vertices of different degrees.

\begin{proposition}\label{prop:K222-Heawood-Petersen}
    We have the following:
    \begin{enumerate}
        \item the complete bipartite graph $K_{2,3}$ is not 0-arc transitive, but TB-symmetrical;
        \item the complete tripartite graph $K_{2,2,2}$ is 1-arc transitive, and not almost-TB-symmetrical;
        \item the Heawood graph is 4-arc transitive, and not TB-symmetrical;
        \item the Petersen graph is 3-arc transitive, and TB-symmetrical;
        \item the cube graph $Q_3$ is almost-TB-symmetrical, but not TB-symmetrical.
    \end{enumerate}
\end{proposition}

Combining Theorem \ref{thm:2-arc} and Proposition~\ref{prop:K222-Heawood-Petersen}, we obtain the following diagram, relating the new concepts with the classical ones:

\begin{center}
         \begin{tikzcd}
    \text{4-arc transitive} \arrow[d,Rightarrow] \arrow[drr, Rightarrow, degil, "\text{the Heawood graph}"] && \\
    \text{3-arc transitive} \arrow[d,Rightarrow] && \text{TB-symmetrical} \arrow[d,Rightarrow] \arrow[dddll, Rightarrow, bend left=130,degil, "K_{2,3}"]   \\
    \text{2-arc transitive}  \arrow[d,Rightarrow]  \arrow[rr,Rightarrow, "Theorem~\ref{thm:2-arc}"] &&\text{almost-TB-symmetrical} \arrow[u, Rightarrow, bend left, degil, "Q_3"] \\
    \text{1-arc transitive} \arrow[d,Rightarrow] \arrow[urr, Rightarrow, degil, "K_{2,2,2}"] && \\
    \text{0-arc transitive}  &&
    \end{tikzcd}
    \end{center}

We have found more almost-TB-symmetrical graphs with Theorem~\ref{thm:2-arc}, but how about TB-symmetrical graphs? It is unclear to us when the third condition in Definition~\ref{def:TB-symmetrical} is satisfied in general. A source of 2-arc transitive graphs (and hence almost-TB-symmetrical) are the \emph{odd graphs}  $O_n$ for $n\geq 2$ (which are in fact 3-arc transitive \cite{MR1373683}, see also \cite[Lemma 4.1.2]{MR1829620}). Recall that for each integer $n$, the odd graph $O_n$ is defined to be the graph with 
\begin{align*}
    V(O_n)&= \{ \text{subsets of $\{1,2,\dots,2n-1\}$ of cardinality $n-1$}\},\\
    E(O_n)&= \{A,B\in V(O_n)\colon A\cap B = \emptyset\}.
\end{align*}
In particular, $O_n$ has $\binom{2n-1}{n-1}$ vertices and $n\binom{2n-1}{n-1}/2$ edges. For example, $O_2$ is $K_3$, and $O_3$ is better known as the Petersen graph. Since both $O_2$ and $O_3$ are TB-symmetrical, we expect  that all odd graphs are TB-symmetrical (see Conjecture~\ref{conj:On}).

\section{A complete search of (almost-)TB-symmetrical graphs with low number of vertices.}

In this section we perform a search for all (almost-)TB-symmetrical graphs with at most 9 vertices. Before doing so, we introduce some graph operations that result in (almost-)TB-symmetrical graphs.

Let $G$ be a graph and $v$ a vertex of $G$. A vertex that forms an edge with $v$ is called its \emph{neighbor}. The vertex $v$ is called a \emph{pendant vertex} if it has exactly one neighbor. Let $G_v$ denote the graph with the vertex set $V(G)\sqcup \{u\}$ and the edge set $E(G)\sqcup \{\{v,u\}\}$. Pictorially, $G_v$ is $G$ with a new pendant vertex attached to a fixed vertex $v$ of $G$. Observe that the edge $\{v,u\}$ is not in any cycle of $G_v$. In other words, no new cycles are created. Thus we have the following.

\begin{lemma}\label{lem:add-pendant}
    Let $G$ be a graph, $v\in V(G)$, and two distinct integers $r,s\geq 3$. Then  $G$ is (almost-)$(r,s)$-TB-symmetrical if and only if $G_v$ is, in which case $\rho_{r,s}(G_v)=\rho_{r,s}(G)$.
\end{lemma}

Let $G_1$ and $G_2$ be two graphs with no common vertices. The \emph{disjoint union} of $G_1$ and $G_2$, denoted by $G_1\cup G_2$, is the graph with both the vertex set and edge set being the corresponding unions of those of $G_1$ and $G_2$.
We have the following.
\begin{lemma}\label{lem:clique-sum}
    Let $G_1$ and $G_2$ be two graphs with no common vertex. If $G_1$ and $G_2$ are (almost-)$(r,s)$-TB-symmetrical for some distinct integers  $r,s\geq 3$, then $G_1 \cup G_2$ is (almost-)$(r,s)$-TB-symmetrical if and only if $\rho_{r,s}(G_1)=\rho_{r,s}(G_2)$, in which case $\rho_{r,s}(G_1 \cup G_2)=\rho_{r,s}(G_1)=\rho_{r,s}(G_2)$.
\end{lemma}
\begin{proof}
    By construction, any cycle in $G_1 \cup G_2$ is a cycle of either $G_1$ or $G_2$. In particular, any pair of one edge from $G_1$ and one edge from $G_2$ does not appear in any cycle of $G_1\cup G_2$. The result then follows.
\end{proof}

It is in fact straightforward to come up with more operations such that the \Cref{lem:clique-sum} and similar arguments in its proof hold. Again let $G_1$ and $G_2$ be two graphs with no common vertices, and $v_1\in V(G_1)$ and $v_2\in V(G_2)$. Let $t\geq 1$ be an integer. The \emph{$t$-path-join} of $G_1$ and $G_2$ at $v_1$ and $v_2$, denoted by $G_1 \tensor[_{v_1}]{\oplus}{^t_{v_2}} G_2$, is the graph obtained by adding $t$ vertices $u_1,\dots, u_t$ and edges $\{v_1,u_1\},\{u_1,u_2\},\dots, \{u_{t-1},u_t\},\{u_t,v_2\}$  to $G_1\cup G_2$. Pictorially, $G_1 \tensor[_{v_1}]{\oplus}{^t_{v_2}} G_2$ is exactly the union $G_1\cup G_2$ together with a path of length $n$ connecting $v_1$ and $v_2$. For this reason, we let $G_1 \tensor[_{v_1}]{\oplus}{^0_{v_2}} G_2$ denote the \emph{1-clique sum} of $G_1$ and $G_2$, which is $G_1\cup G_2$ with $v_1$ and $v_2$ identified as one vertex.

The following is an analog of \Cref{lem:clique-sum}, with a proof following from similar arguments.

\begin{lemma}\label{lem:sum}
    Let $G_1$ and $G_2$ be two graphs with no common vertex, $v_1\in V(G_1)$ and $v_2\in V(G_2)$, and $t\geq 0$ an integer. If $G_1$ and $G_2$ are (almost-)$(r,s)$-TB-symmetrical for some $r\geq s\geq 3$, then $G_1 \tensor[_{v_1}]{\oplus}{^t_{v_2}} G_2$ is (almost-)$(r,s)$-TB-symmetrical if and only if $\rho_{r,s}(G_1)=\rho_{r,s}(G_2)=\rho$, in which case $\rho_{r,s}(G_1 \tensor[_{v_1}]{\oplus}{^t_{v_2}} G_2)=\rho$.
\end{lemma}

Observe that if $G_1$ and $G_2$ are complete graphs, then $G_1 \tensor[_{v_1}]{\oplus}{^t_{v_2}} G_2$ is independent of the chosen vertices $v_1$ and $v_2$. Thus we will use the notation $K_m \oplus^t K_n$ for $K_m \tensor[_{v_1}]{\oplus}{^t_{v_2}} K_n$.

We have obtained many operations that result in (almost-)TB-symmetrical graphs, with the input being (almost-)TB-symmetrical to begin with. 

All almost-TB-symmetrical graphs with at most 9 vertices are listed in the next result. Note that due to Lemmas~\ref{lem:add-pendant} and \ref{lem:clique-sum}, we only consider connected graphs without pendant vertices. Note that we cannot use \cite{code} to verify some of the listed graphs below to be TB-symmetrical, since they may not be 1-arc transitive (see Remark~\ref{rem:code}).

\begin{theorem}\label{thm:9}
    Let $G$ be a finite simple graph with at most 9 vertices. Assume that $G$ is connected does not have a pendant vertex. Then $G$ is almost-TB-symmetrical graphs if and only if $G$ is among the following graphs:
    \begin{itemize}
        \item graphs with at most 9 vertices and all cycles of the same size;
        \item complete graphs: $K_4, K_5, K_6, K_7, K_8, K_9$;
        \item complete bipartite graphs: $K_{3,3}, K_{3,4}, K_{3,5}, K_{4,4}, K_{3,6}, K_{4,5}$;
        \item path-joins of complete graphs: $K_4 \oplus^0 K_4, K_4 \oplus^1 K_4, K_4 \oplus^2 K_4$;
        \item the cube graph $Q_3$.
    \end{itemize}
    Moreover, excluding the cube graph $Q_3$ which is not TB-symmetrical, all the graphs listed above are also TB-symmetrical.
\end{theorem}
\begin{proof}
     The search for almost-TB-symmetrical graphs is via machine-verification using the Sage code in \cite{code}. We note that in this code, we do not consider graphs with a pendant vertex or all cycles of the same size. On the other hand, except for $Q_3$, all of the listed graphs are TB-symmetrical due to Lemma~\ref{lem:same-size-cycles}, \ref{lem:sum}, and Theorem~\ref{thm:Kn-Kmn}. Finally, $Q_3$ is not TB-symmetrical due to Theorem~\ref{prop:K222-Heawood-Petersen}. This concludes the~proof.
\end{proof}

Finally, we remark that the Petersen graph $O_3$ is currently the only known graph that is TB-symmetrical and cannot be obtained from complete graphs, complete bipartite graphs, and the operations discussed in this section. This makes it an interesting example, and we hope it motivates the readers to attack Conjecture~\ref{conj:On}.

\bibliography{references}

@misc{CD,
      title={Planar Legendrian $\Theta$-graphs}, 
      author={Peter Lambert-Cole and Danielle O'Donnol},
      year={2016},
      eprint={1606.00486},
      archivePrefix={arXiv},
      primaryClass={math.GT},
      url={https://arxiv.org/abs/1606.00486}, 
}

@article {edge-transitive,
    AUTHOR = {Andersen, Lars D{\o}vling and Ding, Song Kang and Sabidussi,
              Gert and Vestergaard, Preben Dahl},
     TITLE = {Edge orbits and edge-deleted subgraphs},
   JOURNAL = {Graphs Combin.},
  FJOURNAL = {Graphs and Combinatorics},
    VOLUME = {8},
      YEAR = {1992},
    NUMBER = {1},
     PAGES = {31--44},
      ISSN = {0911-0119,1435-5914},
   MRCLASS = {05C60},
  MRNUMBER = {1157507},
MRREVIEWER = {William\ L.\ Kocay},
       DOI = {10.1007/BF01271706},
       URL = {https://doi.org/10.1007/BF01271706},
}

@article{Tan,
title = {On the maximal Thurston–Bennequin number of knots and links in spatial graphs},
journal = {Topology and its Applications},
volume = {180},
pages = {132-141},
year = {2015},
issn = {0166-8641},
doi = {https://doi.org/10.1016/j.topol.2014.11.011},
url = {https://www.sciencedirect.com/science/article/pii/S0166864114004453},
author = {Toshifumi Tanaka},
}

@article{baader2009,
  title={Legendrian graphs and quiver mutations},
  author={Baader, Sebastian and Ishikawa, Masaharu},
  journal={Mathematische Annalen},
  volume={344},
  number={2},
  pages={289--314},
  year={2009},
  publisher={Springer}
}

@article{odonnol2015,
  title={Legendrian and transversal graphs},
  author={O'Donnol, Danielle},
  journal={Algebraic \& Geometric Topology},
  volume={15},
  number={4},
  pages={2343--2372},
  year={2015},
  publisher={Mathematical Sciences Publishers}
}

@incollection {MR3654493,
    AUTHOR = {O'Donnol, Danielle and Pavelescu, Elena},
     TITLE = {The total {T}hurston-{B}ennequin number of complete and
              complete bipartite {L}egendrian graphs},
 BOOKTITLE = {Advances in the mathematical sciences},
    SERIES = {Assoc. Women Math. Ser.},
    VOLUME = {6},
     PAGES = {117--137},
 PUBLISHER = {Springer, [Cham]},
      YEAR = {2016},
      ISBN = {978-3-319-34139-2; 978-3-319-34137-8},
   MRCLASS = {57M15 (57M50)},
  MRNUMBER = {3654493},
MRREVIEWER = {Marko\ Kranjc},
       DOI = {10.1007/978-3-319-34139-2\_4},
       URL = {https://doi.org/10.1007/978-3-319-34139-2_4},
}

@book {MR1829620,
    AUTHOR = {Godsil, Chris and Royle, Gordon},
     TITLE = {Algebraic graph theory},
    SERIES = {Graduate Texts in Mathematics},
    VOLUME = {207},
 PUBLISHER = {Springer-Verlag, New York},
      YEAR = {2001},
     PAGES = {xx+439},
      ISBN = {0-387-95241-1; 0-387-95220-9},
   MRCLASS = {05-02 (05C50 05E30)},
  MRNUMBER = {1829620},
MRREVIEWER = {Robin\ J.\ Wilson},
       DOI = {10.1007/978-1-4613-0163-9},
       URL = {https://doi.org/10.1007/978-1-4613-0163-9},
}

@Misc{cycle,
          author = {Weisstein, Eric W.},
          title = {Cycle polynomial},
            year={},
          howpublished = {From MathWorld--A Wolfram Resource. \url{https://mathworld.wolfram.com/CyclePolynomial.html}}
        }

@Misc{arc-transitive,
          author = {Weisstein, Eric W.},
          title = {Arc-Transitive Graph},
            year={},
          howpublished = {From MathWorld--A Wolfram Resource. \url{https://mathworld.wolfram.com/Arc-TransitiveGraph.html}}
        }

@Misc{code,
          author = {Chau, Trung and Shah, Tanushree},
          title = {Sage code for (almost-){TB}-symmetrical graphs},
            year={},
          howpublished = {\url{https://github.com/trungchaumath/trungchaumath.github.io/tree/d1a7506c1c1d76ecce5f96b7871a2959dd153ce0/codes/Thurston-Bennequin-symmetrical%20graphs}}
        }

@article {MR1327328,
    AUTHOR = {Conder, Marston and Morton, Margaret},
     TITLE = {Classification of trivalent symmetric graphs of small order},
   JOURNAL = {Australas. J. Combin.},
  FJOURNAL = {The Australasian Journal of Combinatorics},
    VOLUME = {11},
      YEAR = {1995},
     PAGES = {139--149},
      ISSN = {1034-4942,2202-3518},
   MRCLASS = {05C25 (20B25)},
  MRNUMBER = {1327328},
MRREVIEWER = {William\ L.\ Kocay},
}

@incollection {MR1373683,
    AUTHOR = {Babai, L\'aszl\'o},
     TITLE = {Automorphism groups, isomorphism, reconstruction},
 BOOKTITLE = {Handbook of combinatorics, {V}ol.\ 1,\ 2},
     PAGES = {1447--1540},
 PUBLISHER = {Elsevier Sci. B. V., Amsterdam},
      YEAR = {1995},
      ISBN = {0-444-88002-X},
   MRCLASS = {05C25 (05C60 20B25)},
  MRNUMBER = {1373683},
}

@manual{sagemath,
  Key          = {SageMath},
  Author       = {{The Sage Developers}},
  Title        = {{S}ageMath, the {S}age {M}athematics {S}oftware {S}ystem ({V}ersion 9.8)},
  note         = {{\tt https://www.sagemath.org}},
  Year         = {2023},
}

@article {nauty,
    AUTHOR = {McKay, Brendan D. and Piperno, Adolfo},
     TITLE = {Practical graph isomorphism, {II}},
   JOURNAL = {J. Symbolic Comput.},
  FJOURNAL = {Journal of Symbolic Computation},
    VOLUME = {60},
      YEAR = {2014},
     PAGES = {94--112},
      ISSN = {0747-7171,1095-855X},
   MRCLASS = {05C60 (05C85 68Q25)},
  MRNUMBER = {3131381},
MRREVIEWER = {Bharati\ Rajan},
       DOI = {10.1016/j.jsc.2013.09.003},
       URL = {https://doi.org/10.1016/j.jsc.2013.09.003},
}
\bibliographystyle{alpha}

\end{document}